\documentclass[11pt]{amsart}
\usepackage[a4paper,margin=1in]{geometry}
\usepackage{microtype}
\usepackage[T1]{fontenc}
\usepackage[utf8]{inputenc}
\usepackage{lmodern}
\usepackage{amsthm}
\usepackage{amssymb}
\usepackage{enumitem}
\usepackage{mathtools}
\usepackage{thmtools}
\usepackage[abbrev]{amsrefs}
\usepackage{xcolor}
\usepackage{hyperref}
\hypersetup{
    colorlinks,
    linkcolor={red!50!black},
    citecolor={blue!50!black},
    urlcolor={blue!80!black}
}

\setlist{itemsep=2pt}

\DefineSimpleKey{bib}{arxiveprint}
\DefineSimpleKey{bib}{arxivid}
\DefineSimpleKey{bib}{arxivclass}
\makeatletter
\catcode`\'=11 
  \def\arxiveprint{%
    \resolve@inner{\bib@arxiveprint}
  }
  \def\bib@arxiveprint#1{%
    \begingroup
        #1\relax
        \bib@resolve@xrefs
        \bib@field@patches
        \bib'setup
        \let\PrintPrimary\@empty
        {%
          \IfEmptyBibField{arxivid}{\url{https://arxiv.org/}}
          {%
            \href{https://arxiv.org/abs/\bib'arxivid}{\nolinkurl{arXiv:\bib'arxivid}}%
            \IfEmptyBibField{arxivclass}{}{~\nolinkurl{[\bib'arxivclass]}}
          }
        }\bib'transition
        \setbib@@
    \endgroup
  }
\catcode`\'=12 
\makeatother
\BibSpec{arxiv}{%
    +{}  {\PrintAuthors}                {author}
    +{,} { \textit}                     {title}
    +{.} { }                            {part}
    +{:} { \textit}                     {subtitle}
    +{,} { \PrintContributions}         {contribution}
    +{.} { \PrintPartials}              {partial}
    +{,} { }                            {journal}
    +{}  { \textbf}                     {volume}
    +{}  { \PrintDatePV}                {date}
    +{,} { \issuetext}                  {number}
    +{,} { \eprintpages}                {pages}
    +{,} { }                            {status}
    +{,} { \PrintDOI}                   {doi}
    +{,} { available at \eprint}        {eprint}
    +{,} { available at \arxiveprint}   {arxiveprint}
    +{}  { \parenthesize}               {language}
    +{}  { \PrintTranslation}           {translation}
    +{;} { \PrintReprint}               {reprint}
    +{.} { }                            {note}
    +{.} {}                             {transition}
    +{}  {\SentenceSpace \PrintReviews} {review}
}

\theoremstyle{plain}
\newtheorem{thm}{Theorem}[section]
\newtheorem{cor}[thm]{Corollary}
\newtheorem{lem}[thm]{Lemma}
\newtheorem{prop}[thm]{Proposition}
\theoremstyle{definition}
\newtheorem*{defn}{Definition}
\theoremstyle{remark}
\newtheorem*{rmk}{Remark}

\newcommand{\rk}{\operatorname{rank}}
\newcommand{\fund}[1]{\pi_1(#1)}

\newcommand{\PSL}{\mathrm{PSL}}
\newcommand{\tr}{\operatorname{tr}}
\newcommand{\ZZ}{\mathbb{Z}}
\newcommand{\CC}{\mathbb{C}}
\newcommand{\HH}{\mathbb{H}}
\newcommand{\arccosh}{\operatorname{arccosh}}
\newcommand{\area}{\operatorname{area}}
\newcommand{\sys}{\operatorname{sys}}
\newcommand{\M}{\mathcal{M}}

\DeclarePairedDelimiter\abs{\lvert}{\rvert}%

\begin{document}

\title{3-manifolds of rank 3 have filling links}
\author{William Stagner}
\address{\newline
	Department of Mathematics\newline
	Rice University\newline
	Houston, TX 77005, USA}
\email{\href{mailto:william.stagner@rice.edu}{william.stagner@rice.edu}}

\begin{abstract}
	M. Freedman and V. Krushkal introduced the notion of a ``filling'' link in a 3-manifold: a link $L$ is filling in $M$ if for any spine $G$ of $M$ disjoint from $L$, $\pi_1(G)$ injects into $\pi_1(M \setminus L )$. Freedman and Krushkal show that there exist links in the $3$-torus $T^3$ that satisfy a weaker form of filling, but they leave open the question of whether $T^3$ contains an actual filling link. We answer this question affirmatively by proving in fact that every closed, orientable 3-manifold $M$ with $\rk(\pi_1(M)) = 3$ contains a filling link.
\end{abstract}

\maketitle

\section{Introduction}
To formulate M. Freedman and V. Krushkal's definition of a filling link in \cite{FK21}, we define a \emph{spine} of a 3-manifold $M$ to be a 1-complex $G$ in $M$ of least first betti number such that $\pi_1(G) \to \pi_1(M)$ is surjective. Now for a link $L$ in $M$ and a spine $G$ disjoint from $L$, we may ask if $\pi_1(G) \to \pi_1(M \setminus L)$ is injective. If this holds for \emph{every} spine $G$ disjoint from a fixed link $L$, then $L$ is called a \emph{filling link}.

For any given 3-manifold $M$, it is easy to produce 1-complexes that ``fill'' in this sense --- the core $C$ of a handlebody in any Heegaard splitting of $M$ is a filling 1-complex in the sense that for any spine $G$ of $M$ disjoint from $C$, $\pi_1(G)$ injects into $\pi_1(M\setminus C)$. However, the question becomes much more subtle if we begin to look for filling \emph{links} in $M$. Freedman and Krushkal's definition is motivated in part by several well-known results with the same ``filling'' flavor, such as Bing's characterization\footnote{A closed 3-manifold $M$ is diffeomorphic to $S^3$ if and only if every knot in $M$ contained in a 3-ball.} \cite{Bin58} of $S^3$ or Myers' ``disk busting curves'' theorem\footnote{Every compact 3-manifold $M$ contains a knot that intersects any essential disk or sphere in $M$.} \cite{Mye82}. These results exemplify the frequently encountered difficulty of ``upgrading'' essentially trivial statements about 1-complexes to much deeper versions about knots or links.

It is possible to prove the existence of filling links in certain cases. Recall the \emph{rank} of a group is the smallest cardinality of a generating set. We define the rank of a 3-manifold $M$ to be the rank of $\pi_1(M)$. Finding filling links is trivial for 3-manifolds of rank 1, and in the appendix to \cite{FK21}, Leininger and Reid prove the existence of filling links in 3-manifolds of rank 2. For higher rank, establishing the existence of filling links turns out to be not only difficult in general but quite subtle even for concrete examples. Evidence of this is shown in \cite{FK21}, which focuses on the 3-torus $T^3$. In \cite{FK21}, utilizing the augmentation ideal, links $L$ are constructed in $T^3$ that are ``nearly'' filling in the sense that $\pi_1(G) \to \pi_1(T^3 \setminus L)$ is injective modulo (arbitrarily large) terms in the lower central series. However, the question of whether $T^3$ contains an actual filling link is left open.

Our main theorem answers this question in the affirmative.

\begin{restatable}{thm}{mainthm}\label{thm:Main}
	Let $M$ be a closed, orientable 3-manifold of rank 3. Then $M$ contains a filling link.
\end{restatable}

\begin{cor}
	The 3-torus $T^3$ contains a filling link.
\end{cor}

We prove Theorem~\ref{thm:Main} in three main steps. In Section 2, we show that if $L$ is a link in $M$ of rank 3 such that $M \setminus L$ admits a complete hyperbolic metric of finite volume, then either $L$ is a filling link or $\pi_1(M\setminus L)$ contains the fundamental group of a closed surface of genus 2. In Section 3, we show that if the length of the shortest closed geodesic in a finite volume hyperbolic 3-manifold $N$ is sufficiently large, then $\pi_1 (N)$ cannot contain such a subgroup. In Section 4, we construct in $M$ a hyperbolic link that achieves this lower bound on the length of the shortest closed geodesic. Our arguments utilize techniques from minimal surfaces in 3-manifolds, universal links, and Kleinian groups.

\subsection*{Acknowledgements} The author foremost wishes to thank his Ph.D. advisor, Alan Reid, for his many years of mentorship and support, and without whom this paper would not be possible. The author would also like to thank Charles Ouyang, Andrea Tamburelli, and Zeno Huang for several helpful conversations concerning minimal surfaces. Additionally, the author would like to thank Christopher Leininger for helpful comments on an early draft of this paper.

\section{Filling link obstructions}
\subsection{Notation and background}

We work in the upper half-space model of hyperbolic 3-space, $$\HH^3 = \{(z,t) : z\in \CC, t > 0\}, \quad ds^2 = \frac{\abs{dz}^2 + dt^2}{t^2}.$$ In this model, the \emph{boundary at infinity} of $\HH^3$ is identified with the Riemann sphere $\hat \CC = \CC \cup \{\infty\}$. An orientation-preserving isometry of $\HH^3$ extends uniquely to a conformal automorphism of $\hat \CC$, so we identify $\operatorname{Isom}^+(\HH^3) \cong \PSL_2(\CC)$.

A \emph{Kleinian group} $\Gamma$ is a discrete subgroup of $\PSL_2(\CC)$. We consider only torsion-free, non-abelian Kleinian groups, so $\HH^3/\Gamma$ is a complete hyperbolic 3-manifold. The \emph{domain of discontinuity} $\Omega(\Gamma)$ of $\Gamma$ is the largest open subset of $\hat \CC$ on which $\Gamma$ acts properly discontinuously. There is an associated \emph{Kleinian manifold}
$$
	\M(\Gamma) = (\HH^3 \cup \Omega(\Gamma))/\Gamma
$$
with boundary $\partial \M(\Gamma) = \Omega(\Gamma)/\Gamma$, called the \emph{conformal boundary} of $\Gamma$. Note that the complete hyperbolic manifold $ \HH^3 / \Gamma $ is diffeomorphic to the interior $\operatorname{int}(\M(\Gamma))$. We say that $\Gamma$ is \emph{geometrically finite} if the action of $\Gamma$ on $\HH^3$ has a finite-sided fundamental domain.

An isometry $\gamma\in \Gamma \subset \PSL_2(\CC)$ is \emph{parabolic} if $\gamma$ has precisely one fixed point on $\hat \CC$, called a cusp. Otherwise, $\gamma$ is \emph{loxodromic} (since $\Gamma$ is torsion-free). A loxodromic element $\gamma$ has a complex length $$\ell(\gamma) = \ell_0(\gamma) + \theta(\gamma),$$ where $\ell_0(\gamma) > 0$ is the length of the geodesic loop in the free homotopy class corresponding to $\gamma$ in $\HH/\Gamma$. The free homotopy class of a parabolic contains arbitrarily short loops. Note that $\gamma$ is parabolic if and only if $\abs{\tr \gamma}= 2$.

A \emph{hyperbolic link} is a link $L$ in a 3-manifold $M$ such that $M\setminus L$ admits a complete hyperbolic metric of finite volume.

The main result of this section is the following algebraic characterization of a hyperbolic link which is not filling.

\begin{prop}\label{prop:alternative}
	Let $M$ be a closed, orientable 3-manifold of rank $3$, and let $L\subset M$ be a hyperbolic link with at least 4 components. Then either $L$ is filling or $\pi_1(M\setminus L)$ contains a genus 2 surface subgroup.
\end{prop}

\begin{rmk}
	As described in \cite{FK21}*{Appendix}, it is always possible to find a hyperbolic link in a compact, orientable 3-manifold with any prescribed number of components by appealing to \cite{Mye82}.
\end{rmk}
The proof of Proposition~\ref{prop:alternative} will comprise \S 2.2 and \S2.3. For the remainder of this section, we fix a hyperbolic link $L\subset M$ with at least 4 components and an embedding of a spine $i \colon G \hookrightarrow M \setminus L$. Denote the image of the induced map on fundamental groups by $\Gamma = i_*(\pi_1 G) \subset \pi_1(M\setminus L)$.

\subsection{Virtual fiber --- geometrically finite dichotomy}
Since $M\setminus L$ is hyperbolic, the Tameness Theorem of Agol~\cite{Ago04} and Calegari-Gabai~\cite{CG06} together with Canary's covering theorem~\cite{Can96} imply (see, e.g. \cite{AFW15}) that either:
\begin{enumerate}
	\item $\Gamma$ is a virtual surface fiber group, i.e. there exists a finite cover $\widetilde M \to M\setminus L$ and a surface bundle $\Sigma \hookrightarrow \widetilde M \twoheadrightarrow S^1$ with $\Gamma = \pi_1(\Sigma)$, or
	\item $\Gamma$ is geometrically finite.
\end{enumerate}
In the first case, since $\widetilde M$ is non-compact, the fiber surface $\Sigma$ is also non-compact. In particular, $\Gamma = \pi_1(\Sigma) \cong F(n)$ is free, so $\pi_1(\widetilde M)$ decomposes as $$\fund{\widetilde M} \cong F(n) \rtimes Z.$$ However, since $L$ has at least 4 components, we have $$b_1(\widetilde M)\geq b_1(M\setminus L) \geq 4,$$ so $n = 3$ and $\Gamma \cong F(3) \cong \pi_1(G)$.
It follows from the Hopfian property of free groups that $\fund{G} \to \fund{M\setminus L}$ is injective. Hence we see that if $L$ is not filling, then $\Gamma$ is geometrically finite.

\subsection{Geometrically finite case} Now suppose $\Gamma$ is geometrically finite.
If $\Gamma$ contains parabolics, denote by $\M_0(\Gamma)$ the manifold obtained by removing a maximal collection of disjoint cusp neighborhoods from $\M(\Gamma)$. If $\Gamma$ is parabolic-free, then we set $\M_0(\Gamma) = \M(\Gamma)$. It is a classical result of Marden~\cite[Proposition 4.2]{Mar74} that $\Gamma$ is geometrically finite if and only if $\M_0(\Gamma)$ is compact. When $\Gamma$ is understood, we write more simply $\M_0 = \M_0(\Gamma)$.

Since $b_1(\Gamma) \leq 3$, half-lives-half-dies (see e.g., \cite[Lemma 3.5]{Hat}) implies the total genus of $\partial \M_0$ is at most 3. We now analyze the possibilities for $\partial \M_0$.

\emph{$\partial \M_0$ is a union of incompressible tori.} In this case, $\operatorname{int}(\M_0) = \HH^3/\Gamma$ has finite volume. Hence the inclusion $\Gamma \subset \pi_1(M\setminus L)$ induces a covering $\HH^3/\Gamma \to M\setminus L$ which is finite-sheeted. But recall that $M\setminus L$ has at least 4 cusps, while $\HH^3/\Gamma$ has at most 3 cusps, so we reach a contradiction.

\emph{$\partial \M_0$ contains an incompressible component of genus 2.} We have immediately that $\Gamma$, and hence also $\pi_1(M\setminus L)$, contains a genus 2 surface subgroup, as desired.

\emph{$\partial \M_0$ is incompressible of genus $3$.} We first compute the betti numbers of $\M_0$. Since $\M_0$ is connected with non-empty boundary, we immediately have $b_0(\M_0) = 1$ and $b_3(\M_0) = 0$. To compute $b_1$, first note that $b_1(\M_0) \geq 3$ by half-lives-half-dies. On the other hand, $\Gamma \cong \pi_1(\M_0)$ is rank 3, so $b_1(\M_0) \leq 3$. Hence, $b_1(\M_0) = 3$.

Now recall that $$\chi(\M_0) = \frac{1}{2}\chi(\partial \M_0) = -2,$$ which follows from an elementary application of the Mayer-Vietoris sequence. From this we obtain $b_2(\M_0) = 0$.

Suppose now $\Gamma = \langle \gamma_1, \gamma_2, \gamma_3 \rangle$. Since $b_1(\Gamma) = 3$, the images of $\gamma_1, \gamma_2, \gamma_3$ in $H_1(\Gamma; \mathbb Q)$ are linearly independent over $\mathbb Q$. This together with $H_2(\Gamma; \mathbb Q) = 0$ allows us to apply \cite{Sta65}*{Theorem 7.4}, which states that $\gamma_1, \gamma_2, \gamma_3$ freely generate a subgroup of $\Gamma$, which of course is equal to $\Gamma$. But this yields a contradiction, since $\Gamma$ contains a non-free subgroup induced by $\partial \M_0$, which was assumed to be incompressible of genus 3.

\emph{$\partial \M_0$ contains a compressible component.} Let $D\subset \partial \M_0$ be a compressing disk and $Y$ the manifold obtained by cutting along $D$. If $D$ is non-separating, then $$\Gamma \cong \pi_1 (Y) * \ZZ$$ with $\rk(\pi_1 (Y))=2$. On the other hand, if $D$ is separating with components $Y'$, $Y''$, then by Grushko's theorem, and without loss of generality, we can assume that $\rk(\pi_1(Y')) = 2$, $\rk(\pi_1(Y'')) = 1$, so that $$\Gamma \cong \pi_1(Y') * \pi_1(Y'') \cong \pi_1(Y') * \ZZ.$$ In both cases, we see that $\Gamma$ decomposes as $\Gamma \cong \Gamma' * \ZZ$, for some two-generator subgroup $\Gamma'$.

Now by \cite{JS79}, we have the following possibilities for $\Gamma'$:
\begin{enumerate}
	\item $\Gamma'$ has finite index in $\pi_1(M\setminus L)$, or
	\item $\Gamma'$ is free of rank 2, or
	\item $\Gamma'$ is free-abelian of rank 2.
\end{enumerate}

Clearly case (1) is impossible since $b_1(\pi_1(M \setminus L)) \geq 4$. 

For case (2), if $\Gamma'$ is free of rank 2, then $\Gamma$ is free of rank 3, so $\pi_1(G) \to \pi_1(M\setminus L)$ is injective. 

Now we consider case (3). Let $\phi\colon \pi_1(M\setminus L) \to \pi_1(M)$ be the map on $\pi_1$ induced by trivially filling each component of $L$. Recall that $\ker(\phi)$ is normally generated by a collection of meridians $\{\mu_1,\dots,\mu_r\}$, where each $\mu_i$ is contained in a peripheral subgroup $P_i$ corresponding to a component of $L$. Since $M\setminus L$ is atoroidal, the rank 2 abelian subgroup $\Gamma' \cong \ZZ\times \ZZ$ is conjugate into some peripheral subgroup $P_j$ of $\pi_1(M\setminus L)$. After choosing a framing $P_j = \langle \mu_j, \lambda_j \rangle$, we see that $\rk (\phi(P_j)) = \rk(\langle \phi(\lambda_j)\rangle) \leq 1$. In particular, $\rk(\phi(\Gamma')) \leq 1$, so $\rk(\phi(\Gamma)) \leq 2$. But recall that $G$ is a spine of $M$, so the composition $$ \pi_1(G) \xrightarrow{i_*} \pi_1(M\setminus L) \xrightarrow{\phi} \pi_1(M)$$ is surjective by definition. In particular, $$\rk(\phi(i_*(G))) = \rk(\phi(\Gamma)) = 3,$$ which is a contradiction. Thus, we exclude the possibility of case (3). This completes the proof of Proposition~\ref{prop:alternative}.\qed

\section{Genus bounds of essential surfaces}\label{section:min-surfaces}

We now use tools from minimal surface theory to translate the algebraic obstruction of Proposition~\ref{prop:alternative} into a geometric one. First we recall some basic facts about minimal surfaces.

\subsection{Minimal surfaces} Recall that a \emph{minimal surface} in a 3-manifold is an immersed surface with mean curvature identically zero. Equivalently, a minimal surface is a critical point of the area functional. A surface which is a local minimum of the area functional is called \emph{stable}. We will use the following area estimate which follows from the monotonicity formula (see e.g. \cite{And82}).

\begin{lem}\label{monotonicity}
	Let $M$ be a hyperbolic 3-manifold and let $B(x,r)$ be the open ball of radius $r \leq \operatorname{injrad}_M(x)$ centered at $x$. If $S$ is a stable minimal surface passing through $x$, then
	\begin{equation}
		\area(S\cap B(x,r)) \geq 2\pi(\cosh r - 1).
	\end{equation}
\end{lem}
There is an extensive body of work constructing stable minimal surfaces which are area-minimizing in their homotopy classes, based on the seminal work of Schoen and Yau \cite{SY79}, Sacks and Uhlenbeck \cites{SU81,SU82}, and many others \cites{MY82, HS88}. We make use of the following result of Huang and Wang \cite[Theorem 1.1]{HW17}.
\begin{thm}\label{thm:least-area-exist}
	Let $S$ be a closed orientable surface of genus at least two, which is immersed in a cusped hyperbolic 3-manifold $M^3$. If $S$ is $\pi_1$-injective, then $S$  is homotopic to an immersed least-area minimal surface in $M$.
\end{thm}
The next lemma, attributed to Uhlenbeck \cite[Lemma 6]{Has95}, gives a lower bound for the genus of a least-area surface in terms of its area.

\begin{lem}\label{lem:genus-area-bound}
	A closed, $\pi_1$-injective, least-area surface $S$ of genus $g$ immersed in a hyperbolic 3-manifold satisfies
	\begin{equation}
		g \geq \frac{\area(S)}{4\pi} + 1.
	\end{equation}
\end{lem}
Now we relate the area of minimal surfaces to the geometry of the ambient 3-manifold. Recall that an \emph{accidental parabolic} of an immersed closed surface $i\colon S \looparrowright M$ is a simple closed curve $\alpha \subset S$ such that $i_*(\alpha) \in \pi_1(M)$ is parabolic.

\begin{prop}\label{prop:area-systole-bound}
	Let $M$ be a complete hyperbolic 3-manifold and $S$ a closed, stable minimal surface immersed in $M$ with no accidental parabolics. Then the area of $S$ satisfies
	\begin{equation}
		\area(S) \geq 2\pi \left(\cosh\left(\frac{\sys(M)}{2}\right) - 1\right),
	\end{equation}
	where $\sys(M)$ denotes the length of the shortest closed geodesic in $M$.
\end{prop}

\begin{proof}
	Let $p\colon \HH^3 \to M$ be a covering map, and let $\tilde{S}$ be a lift of $S$ to $\HH^3$. Take any point $x\in \tilde{S}$ and let $\tilde{S}_0$ be a connected component of $S \cap B(x, \sys(M)/2)$ containing $x$. We first claim that $\left.p\right|_{\tilde{S}_0}$ is injective. To see this, suppose $x_0, x_1 \in \tilde{S}_0$ are distinct with $p(x_0) = p(x_1)$, and let $\gamma\in \operatorname{Isom}^+(\HH^3)$ be the covering transformation with $\gamma(x_0) = x_1$. Since $\tilde{S}_0$ is connected, let $\alpha\colon [0,1] \to  \tilde{S}_0$ be a path connecting $x_0$ to $x_1$. Then the loop $p\circ \alpha$ lies in the free homotopy class $\left[\gamma\right]\in \pi_1(M)$. In particular, $p\circ \alpha$ is essential in $M$, and by incompressibility, also essential in $S$. Now homotope $\alpha$ rel endpoints to the unique geodesic arc $\alpha'$ connecting $x_0$ and $x_1$. Composing with $p$ yields a homotopy in $M$ from $p\circ \alpha$ to $p\circ \alpha'$, which has length $\ell_0(p\circ \alpha') = d_{\HH_3}(x_0, x_1) < \sys(M)$. Thus $p\circ \alpha$ is an accidental parabolic in $S$, a contradiction.
	
	Now we have $\left.p\right|_{\tilde{S}_0}\colon \tilde{S}_0 \to S$ is an isometric embedding, so in particular, $\area(S) \geq \area(\tilde{S}_0)$. The proposition follows immediately by applying Lemma~\ref{monotonicity}.
\end{proof}

The previous proposition in essence says that surfaces (without accidental parabolics) in $M$ must be sufficiently complex if the geodesics in $M$ are not too short. We now prove this unconditionally for genus 2 surfaces, with or without accidental parabolics.

\begin{thm}\label{thm:systole-genus-bound}
	Let $M$ be a finite volume hyperbolic 3-manifold and $S$ a closed, $\pi_1$-injective surface of genus $g>1$ immersed in $M$. If $\sys(M) > 2\arccosh(9/2) = 4.369...$, then $g \geq 3$.
\end{thm}

\begin{proof}
	By Theorem~\ref{thm:least-area-exist} we may after a homotopy assume $S$ is least-area. Note that least-area minimal surfaces are stable. If $S$ has no accidental parabolics, then Proposition~\ref{prop:area-systole-bound} with $\sys(M) > 2\arccosh(9/2)$ gives $\area(S) \geq 7\pi$. Immediately by Lemma~\ref{lem:genus-area-bound} we have $g\geq 3$.
	
	Now suppose $g=2$ and $S$ has an accidental parabolic. We will show that this implies $\sys(M) \leq 2\arccosh(9/2)$. Since $\pi_1 (M)$ is LERF (see e.g. \cite[Chapter 5.2, H.11]{AFW15} for details), we pass to a finite cover $\tilde M$ in which $S$ embeds. Additionally, we truncate $\tilde M$ by removing cusp neighborhoods disjoint from $S$, and so we may take $\tilde M$ to be compact with toroidal boundary. Now take a simple closed curve in $S$ representing the accidental parabolic. By the annulus theorem, this curve cobounds an annulus with another simple closed curve in $\partial \tilde M$. Now we compress $S$ to $\partial \tilde M$ along this annulus. The result is either two once-punctured tori or a single twice-punctured torus, depending on whether the compressing annulus separates $S$. 
	
	In the first case, the once-punctured tori are boundary incompressible, otherwise cutting along a boundary compressing disk yields an essential annulus, which is impossible since $\tilde M$ is hyperbolic. Now consider the second case, so $\tilde M$ contains a properly embedded twice-punctured torus. If this surface is boundary compressible, then again cutting along a compressing disk yields either a once-punctured torus and essential annulus, which is impossible, or a single once-punctured torus, which by the above must be boundary incompressible. Hence $\tilde M$ contains either a once or twice-punctured torus which is incompressible and boundary incompressible. With this, we apply Theorem 5.2 of \cite{AR00} to obtain
	\begin{equation*}
		\sys(\tilde M) \leq \max\{2 \arccosh 3, \min\{2\arccosh(9/2), 4\arccosh(3) \} \} = 2\arccosh(9/2).
	\end{equation*}
	But since $\tilde M$ finitely covers $M$, we have $\sys(M) \leq \sys(\tilde M) \leq 2\arccosh(9/2)$, as desired.
\end{proof}
\section{A universal congruence link}

\subsection{Principal congruence manifolds}

Denote by $\mathcal O_d$ the ring of integers of the number field $\mathbb Q(\sqrt{-d})$ for $d>0$ a square-free integer. This gives rise to a family of finite-covolume Kleinian groups $\Gamma_d = \PSL_2(\mathcal O_d)$ called the \emph{Bianchi groups}. Note that $\mathcal O_1 = \ZZ[i]$, so for $d=1$ we will simply write $\Gamma_1 = \PSL_2(\ZZ[i])$. The Bianchi groups arise as the natural 3-dimensional analogue of the modular group $\PSL_2(\ZZ)$: the non-cocompact arithmetic Kleinian groups are precisely the Kleinian groups commensurable (in the wide sense) with some $\Gamma_d$ (see \cite{MR03} for background). To an ideal $I < \mathcal O_d$, there is an associated \emph{principal congruence group} $$\Gamma(I) = \ker \left\{ \PSL_2(\mathcal O_d) \to \PSL_2(\mathcal O_d/I) \right\}.$$ The quotient $\HH^3/\Gamma(I)$ is called a \emph{principal congruence manifold}. When $I = \langle \alpha \rangle$ is principal, we write $\Gamma(\alpha) = \Gamma(\langle \alpha\rangle)$. The next proposition utilizes the algebra of principal congruence subgroups to get good control over the geometry of their manifold quotients.

\begin{prop}\label{prop:cong-systole-bound}
	Let $\langle \alpha \rangle < \mathcal O_d$ be a principal ideal. If $\lvert \alpha \rvert > \sqrt{11}$, then $$\sys(\HH^3/\Gamma(\alpha) > 2\arccosh(9/2).$$
\end{prop}

\begin{proof}
	Let $\gamma\in \Gamma(I)$ be loxodromic, and write $$ \gamma = \pm\begin{pmatrix}1+a & b \\ c & 1 + d \end{pmatrix}, \quad a,b,c,d \in I.$$
	Since $\det \gamma = 1$, we have
	$$
		\pm\tr \gamma = 2 + (bc - ad) \equiv 2 \pmod {I^2}.
	$$
	Hence, for a principal ideal $I = \langle \alpha \rangle$, we may write $\pm \tr \gamma = z\alpha^2 + 2$ for some $z\in \mathcal O_d$. Then the reverse triangle inequality yields
	\begin{equation}
		\lvert \tr\gamma\rvert \geq \big\lvert \lvert z\rvert\lvert \alpha\rvert^2 - 2\big\rvert > \lvert11-2\rvert = 9.
	\end{equation}
	Now recall that the complex translation length $\ell(\gamma) = \ell_0(\gamma) + i\theta(\gamma)$ satisfies $$\tr\gamma = 2\cosh\left(\ell(\gamma)/2\right).$$ Hence we obtain $$\ell_0(\gamma) \geq 2\arccosh(\lvert \tr\gamma\rvert/2) > 2\arccosh(9/2),$$ as desired.
\end{proof}

For the remainder of the paper, we will narrow our focus to a particular principal congruence group $\Gamma(3 + 2i) < \PSL_2(\mathcal \ZZ[i])$. It was shown in \cite{BGR19b} that in fact $\HH^3/\Gamma(3 + 2i)$ is the complement of a link $L_0$ in $S^3$. We will discuss this link further in \S 4.2. For now, note $\lvert 3 + 2i\rvert = \sqrt{13} > \sqrt{11}$.
\begin{cor}\label{cor:no-subgroup}
	The link group $\pi_1(S^3\setminus L_0)$ and all its finite index subgroups do not contain a surface subgroup of genus 2.
\end{cor}
\begin{proof}
	Suppose a genus 2 surface subgroup exists as stated. Since $\pi_1(S^3 \setminus L_0)$ is LERF \cite[Chapter 5.2, H.11]{AFW15}, there exists a finite-sheeted cover of $S^3\setminus L_0$ that contains an closed, embedded, $\pi_1$-injective surface of genus $2$. But now Theorem \ref{thm:systole-genus-bound}, Proposition \ref{prop:cong-systole-bound}, and the discussion preceding the corollary yields a contradiction. Note that systole is non-decreasing in finite covers.
\end{proof}

\subsection{Universal Links}
It is a classical theorem of Alexander~\cite{Ale20} that every closed orientable smooth $n$-manifold is a branched cover of $S^n$. The branch locus in Alexander's theorem is in general only a codimension 2 CW-complex, not a submanifold. However, in the 3-dimensional setting, Hilden~\cite{Hil74}, Hirsch~\cite{Hir74}, and Montesinos~\cite{Mon74} independently proved that there always exists a branched cover $M^3 \to S^3$ whose branch locus is a \emph{knot}. We remark that this is very similar to the present consideration of filling subsets: it is easy to construct a filling complex (e.g. the core of a handlebody in a Heegaard splitting), but it is much more difficult to upgrade such a complex to be a submanifold.

Thurston extended this idea in an unpublished manuscript \cite{Thu82} by constructing a link $L$ in $S^3$ with the remarkable property that \emph{every} closed, orientable 3-manifold is a branched covering of $S^3$ with branch set precisely $L$.

\begin{defn}
	A link $L$ in $S^3$ is called \emph{universal} if every closed, orientable 3-manifold is a branched cover of $S^3$ with branch set $L$.
\end{defn}

The goal of this section is to show that the principal congruence link $L_0$ is universal. Known universal links include the Whitehead link \cite{HLM83}, Borromean rings \cite{HLM83}, and non-torus 2-bridge links \cite{HLM85}. Note that this implies the figure-eight knot is universal, a result which we will use later. In fact, it is an open question whether all hyperbolic links in the 3-sphere are universal,\footnote{In fact, the only known links known \emph{not} to be universal are iterated cables of torus links.} but it is difficult in practice to determine for any given link. The most straightforward strategy is to bootstrap off of other known universal links by utilizing the following basic observation.

\begin{lem}\label{lem:sublink-universal}
	If a link $L\subset S^3$ contains a sublink $L'$ that is universal, then $L$ is universal.
\end{lem}

\begin{proof}
	Let $p\colon M \to S^3$ be a covering branched over $L' \subset S^3$. Recall that $p$ is determined by its monodromy $m\colon \pi_1(S^3 \setminus L') \to S_n$ (see, e.g. \cite[Chapter 10]{Rol76}). We can trivially extend the monodromy to $m\colon \pi_1(S^3 \setminus L) \to S_n$ by mapping each meridian of $L \setminus L'$ to the identity permutation and hence obtain a cover $M\to S^3$ branched over $L$.
\end{proof}

We now seek to find a universal sublink of $L_0$, but our search is complicated by there being no known diagram for $L_0$. We give two computational methods to circumvent this problem: a geometric approach using SnapPy~\cite{SnapPy} and an algebraic one using Magma~\cite{Magma}.

\begin{lem}\label{lem:fig-eight-component}
	The principal congruence link $L_0$ has the figure-eight knot as a component.
\end{lem}

\begin{proof}[Proof using SnapPy]
	We use the triangulation \verb|pSL_3_plus_2_sqrt_minus_1.tri| given in \cite{BGR19a} of $S^3 \setminus L_0$. We claim that the component of $L_0$ corresponding to cusp 20 of this triangulation is the figure-eight knot. To see this, we trivially Dehn fill every other component sequentially, retriangulating after every step.
	\begin{quote}
		\begin{verbatim}
In[1]: M = Manifold("pSL_3_plus_2_sqrt_minus_1.tri")

In[2]: for _ in range(20):
         M.dehn_fill((1,0),0)
         M = M.filled_triangulation()
         # fill the first 20 cusps

In[3]: for _ in range(M.num_cusps() - 1):
         M.dehn_fill((1,0),-1)
         M = M.filled_triangulation()
         # fill the last 21 cusps

In[4]: N = Manifold("4_1")
       # figure-eight knot

In[5]: M.is_isometric_to(N)
Out[5]: True
		\end{verbatim}
	\end{quote}
	Note that filling each component one at a time and calling \verb|M.filled_triangulation| afterwards is essential to avoid floating-point errors.
\end{proof}

\begin{proof}[Proof using Magma]
	Here we take a group-theoretic approach to Dehn filling and make extensive use of the methods and subroutines given in \cite{BGR19a}. We use the following presentation of $\PSL_2(\ZZ[i])$ given in \cite{Swa71}: $$ \PSL_2(\ZZ[i]) = \langle a,\ell,t,u \mid \ell^2=(t\ell)^2=(u\ell)^2=(a\ell)^2=a^2=(ta)^3=(ua\ell)^3=1, [t,u] = 1 \rangle.$$ The generators correspond to matrices
	\begin{equation*}
		a = \begin{pmatrix} 0 & -1 \\ 1 & 0\end{pmatrix},\quad \ell = \begin{pmatrix} -i & 0 \\ 0 & i\end{pmatrix},\quad
		t = \begin{pmatrix} 1 & 1 \\ 0 & 1\end{pmatrix},\quad u = \begin{pmatrix} 1 & i \\ 0 & 1\end{pmatrix}.
	\end{equation*}
	Since $\Gamma(3 + 2i)$ is isomorphic to a link group $\pi_1(S^3 \setminus L_0)$, we have that $\Gamma(3 + 2i)$ is normally generated by the meridians of $L_0$. Using the method of \cite[Section 13]{BGR19a}, we compute a complete set $P$ of meridians given in Table~\ref{table:meridians}. Each cusp of $\Gamma(3 + 2i)$ corresponds to precisely one meridian in $P$.
	
	To simplify the computation, we note that since $\PSL_2(\ZZ[i])$ has a single cusp, the peripheral subgroups of $\Gamma(3 + 2i)$ are all mutually conjugate in $\PSL_2(\ZZ[i])$; in fact, a computation using the methods of \cite{BGR19a} shows that they are all conjugate to the subgroup $\langle t^{13}, t^{-5}u\rangle$. Hence $\Gamma(3 + 2i)$ is the normal closure of $\langle t^{13}, t^{-5}u\rangle$ in $\PSL_2(\ZZ[i])$. 
	
	We finish by using Magma to kill all but one meridian, resulting in the figure-eight knot group.
	\begin{quote}
		\begin{verbatim}
Bianchi<a,l,t,u> := Group<a,l,t,u|a^2,l^2,(t*l)^2,(u*l)^2,
                         (a*l)^2,(t*a)^3,(u*a*l)^3,(t,u)>;

G := NormalClosure(Bianchi, sub<Bianchi | t^13, t^-5*u>);
G := Rewrite(Bianchi, G : Simplify := false);
// G = Gamma(<3 + 2i>)

Q := quo< G | P[1..13], P[15..42] >;
Q := ReduceGenerators(Q);
// Filling all but cusp 14
// P is the set of meridians in Table 1

H<a,b> := Group<a,b | a^-1*b*a*b^-1*a*b=b*a^-1*b*a>;
// Figure-eight knot group

SearchForIsomorphism(Q,H,7);
// true
		\end{verbatim}
	\end{quote}
	\vspace{-1em}
\end{proof}

Since the figure-eight knot is universal \cite{HLM85}, Lemma~\ref{lem:sublink-universal} and Lemma~\ref{lem:fig-eight-component} immediately imply the main result of this section.

\begin{prop}\label{prop:cong-universal}
	The principal congruence link $L_0$ is universal.
\end{prop}

\subsection{Proof of Theorem~\ref{thm:Main}}
We recall our main theorem from the introduction.
\mainthm*
The theorem follows immediately by combining the following proposition with Proposition~\ref{prop:alternative}.

\begin{prop}\label{prop:genus-2-small-existence}
	Let $M$ be a closed, orientable 3-manifold. Then there exists a hyperbolic link $L\subset M$ with at least 4 components such that $\pi_1(M\setminus L)$ does not contain a genus 2 surface subgroup.
\end{prop}
\begin{proof}
	By Proposition~\ref{prop:cong-universal}, there exists a covering $p\colon M\to S^3$ branched over $L_0$. Since $L_0$ is a $42$-component link, certainly $L = p^{-1}(L_0)$ has at least 4 components. Now $p$ restricts to an unbranched finite-sheeted covering $p \colon M\setminus L \to S^3 \setminus L_0$, so $\pi_1(M\setminus L) \subset \pi_1(S^3 \setminus L_0)$. Hence $\pi_1(M\setminus L)$ does not contain a genus 2 surface subgroup by Corollary~\ref{cor:no-subgroup}.
\end{proof}

\begin{table}
	\caption{The set $P$ of meridians of $L_0$. Here $h = uau^{-2}au^{-1}ta$.\label{table:meridians}}
	\begin{tabular}{l | l}
		\hline
		$ t^{-5}u $                                 & $ ht^8ut^{-5}uh^{-1} $                   \\
		$ h^2t^8ut^{-5}ut^{-5}uh^{-2} $             & $ at^8ut^{-5}ut^{-5}ua^{-1} $            \\
		$ hat^8ut^{-5}ut^{-5}ua^{-1}h^{-1} $        & $ h^2at^8ut^{-5}ua^{-1}h^{-2} $          \\
		$ tat^8ut^{-5}ut^{-5}ua^{-1}t^{-1} $        & $ t^2at^8ut^{-5}ua^{-1}t^{-2} $          \\
		$ t^3at^8ut^{-5}ua^{-1}t^{-3} $             & $ t^4at^8ut^{-5}ua^{-1}t^{-4} $          \\
		$ t^5at^8ut^{-5}ua^{-1}t^{-5} $             & $ t^6at^8ut^{-5}ua^{-1}t^{-6} $          \\
		$ t^{-1}at^{-5}ua^{-1}t $                   & $ t^{-2}at^8ut^{-5}ut^{-5}ua^{-1}t^2 $   \\
		$ t^{-3}at^8ut^{-5}ut^{-5}ua^{-1}t^3 $      & $ t^{-4}at^8ut^{-5}ut^{-5}ua^{-1}t^4 $   \\
		$ t^{-5}at^8ut^{-5}ua^{-1}t^5 $             & $ t^{-6}at^8ut^{-5}ua^{-1}t^6 $          \\
		$ that^8ut^{-5}ut^{-5}ua^{-1}h^{-1}t^{-1} $ & $ t^2hat^8ut^{-5}ua^{-1}h^{-1}t^{-2} $   \\
		$ t^3hat^8ut^{-5}ua^{-1}h^{-1}t^{-3} $      & $ t^4hat^8ut^{-5}ua^{-1}h^{-1}t^{-4} $   \\
		$ t^5hat^8ut^{-5}ua^{-1}h^{-1}t^{-5} $      & $ t^6hat^8ut^{-5}ua^{-1}h^{-1}t^{-6} $   \\
		$ t^{-1}hat^8ut^{-5}ua^{-1}h^{-1}t $        & $ t^{-2}hat^8ut^{-5}ua^{-1}h^{-1}t^2 $   \\
		$ t^{-3}hat^8ut^{-5}ua^{-1}h^{-1}t^3 $      & $ t^{-4}hat^8ut^{-5}ua^{-1}h^{-1}t^4 $   \\
		$ t^{-5}hat^8ut^{-5}ua^{-1}h^{-1}t^5 $      & $ t^{-6}hat^8ut^{-5}ua^{-1}h^{-1}t^6 $   \\
		$ th^2at^8ut^{-5}ua^{-1}h^{-2}t^{-1} $      & $ t^2h^2at^8ut^{-5}ua^{-1}h^{-2}t^{-2} $ \\
		$ t^3h^2at^8ut^{-5}ua^{-1}h^{-2}t^{-3} $    & $ t^4h^2at^8ut^{-5}ua^{-1}h^{-2}t^{-4} $ \\
		$ t^5h^2at^8ut^{-5}ua^{-1}h^{-2}t^{-5} $    & $ t^6h^2at^8ut^{-5}ua^{-1}h^{-2}t^{-6} $ \\
		$ t^{-1}h^2at^8ut^{-5}ua^{-1}h^{-2}t $      & $ t^{-2}h^2at^8ut^{-5}ua^{-1}h^{-2}t^2 $ \\
		$ t^{-3}h^2at^8ut^{-5}ua^{-1}h^{-2}t^3 $    & $ t^{-4}h^2at^8ut^{-5}ua^{-1}h^{-2}t^4 $ \\
		$ t^{-5}h^2at^8ut^{-5}ua^{-1}h^{-2}t^5 $    & $ t^{-6}h^2at^8ut^{-5}ua^{-1}h^{-2}t^6 $ \\
		\hline
	\end{tabular}
\end{table}


\begin{bibdiv}
	\begin{biblist}

		\bib{AR00}{article}{
			author={Adams, Colin},
			author={Reid, Alan~W.},
			title={Systoles of hyperbolic 3-manifolds},
			date={2000-01},
			journal={Math. Proc. Camb. Philos. Soc.},
			volume={128},
			pages={103\ndash 110},
		}
		
		\bib{Ago04}{arxiv}{
			author={Agol, Ian},
			title={Tameness of hyperbolic 3-manifolds},
			date={2004-05},
			arxiveprint={
					arxivid={math/0405568},
					arxivclass={math.GT}},
		}
		
		\bib{Ale20}{article}{
			author={Alexander, James~W.},
			title={Note on {{Riemann}} spaces},
			date={1920},
			ISSN={0002-9904},
			journal={Bull. Amer. Math. Soc.},
			volume={26},
			number={8},
			pages={370\ndash 372},
		}
		
		\bib{And82}{article}{
			author={Anderson, Michael~T.},
			title={Complete minimal varieties in hyperbolic space},
			date={1982},
			ISSN={0020-9910},
			journal={Invent. Math.},
			volume={69},
			number={3},
			pages={477\ndash 494},
		}
		
		\bib{AFW15}{book}{
			author={Aschenbrenner, Matthias},
			author={Friedl, Stefan},
			author={Wilton, Henry},
			title={3-manifold groups},
			series={{{EMS}} Series of Lectures in Mathematics},
			publisher={{European Mathematical Society (EMS), Z\"urich}},
			date={2015},
			ISBN={978-3-03719-154-5},
		}
		
		\bib{BGR19a}{article}{
			author={Baker, M.~D.},
			author={Goerner, M.},
			author={Reid, A.~W.},
			title={All principal congruence link groups},
			date={2019},
			ISSN={0021-8693},
			journal={J. Algebra},
			volume={528},
			pages={497\ndash 504},
		}
		
		\bib{BGR19b}{arxiv}{
			author={Baker, Mark~D.},
			author={Goerner, Matthias},
			author={Reid, Alan~W.},
			title={Technical report: all principal congruence link groups},
			date={2019-02},
			arxiveprint={
					arxivid={1902.04722},
					arxivclass={math.GT}},
		}
		
		\bib{Bin58}{article}{
			author={Bing, R.~H.},
			title={Necessary and sufficient conditions that a 3-manifold be {$S^3$} },
			date={1958},
			ISSN={0003-486X},
			journal={Ann. of Math. (2)},
			volume={68},
			pages={17\ndash 37},
		}
		
		\bib{Magma}{article}{
		author={Bosma, Wieb},
		author={Cannon, John},
		author={Playoust, Catherine},
		title={The {{Magma}} algebra system. {{I}}. {{The}} user language},
		date={1997},
		ISSN={0747-7171},
		journal={J. Symbolic Comput.},
		volume={24},
		number={3-4},
		pages={235\ndash 265},
		note={Computational algebra and number theory (London, 1993)},
		}
		
		\bib{CG06}{article}{
			author={Calegari, Danny},
			author={Gabai, David},
			title={Shrinkwrapping and the taming of hyperbolic 3-manifolds},
			date={2006},
			ISSN={0894-0347},
			journal={J. Amer. Math. Soc.},
			volume={19},
			number={2},
			pages={385\ndash 446},
		}
		
		\bib{Can96}{article}{
			author={Canary, Richard~D.},
			title={A covering theorem for hyperbolic 3-manifolds and its
					applications},
			date={1996},
			ISSN={0040-9383},
			journal={Topology},
			volume={35},
			number={3},
			pages={751\ndash 778},
		}
		
		\bib{SnapPy}{misc}{
			author={Culler, Marc},
			author={Dunfield, Nathan M.},
			author={Goerner, Matthias},
			author={Weeks, Jeffrey R.},
			title={Snap{P}y, a computer program for studying the geometry and topology of $3$-manifolds},
			note={Available at \url{http://snappy.computop.org}}
		}
		
		\bib{FK21}{arxiv}{
			author={Freedman, Michael},
			author={Krushkal, Vyacheslav},
			title={Filling links and spines in 3-manifolds},
			date={2021-02},
			contribution={
					type={appendix},
					author={Leininger, Christopher J.},
					author={Reid, Alan W.}},
			arxiveprint={
					arxivid={2010.15644},
					arxivclass={math.GT}},
		}
		
		\bib{Has95}{article}{
			author={Hass, Joel},
			title={Acylindrical surfaces in 3-manifolds},
			date={1995},
			ISSN={0026-2285},
			journal={Michigan Math. J.},
			volume={42},
			number={2},
			pages={357\ndash 365},
		}
		
		\bib{HS88}{article}{
			author={Hass, Joel},
			author={Scott, Peter},
			title={The existence of least area surfaces in 3-manifolds},
			date={1988},
			ISSN={0002-9947},
			journal={Trans. Amer. Math. Soc.},
			volume={310},
			number={1},
			pages={87\ndash 114},
		}
		
		\bib{Hat}{misc}{
		author={Hatcher, Allen},
		title={Notes on basic 3-{{Manifold}} topology},
		date={2007},
		note={Available at \url{https://pi.math.cornell.edu/~hatcher/3M/3Mdownloads.html}},
		}
		
		\bib{Hil74}{article}{
			author={Hilden, Hugh~M.},
			title={Every closed orientable 3-manifold is a 3-fold branched covering
					space of {$S^3$}},
			date={1974},
			ISSN={0002-9904},
			journal={Bull. Amer. Math. Soc.},
			volume={80},
			pages={1243\ndash 1244},
		}
		
		\bib{HLM83}{article}{
		author={Hilden, Hugh~M.},
		author={Lozano, Mar{\'i}a~Teresa},
		author={Montesinos, Jos{\'e}~Mar{\'i}a},
		title={The {{Whitehead}} link, the {{Borromean}} rings and the knot
				9{\textsubscript{46}} are universal},
		date={1983},
		ISSN={0010-0757},
		journal={Collect. Math.},
		volume={34},
		number={1},
		pages={19\ndash 28},
		}
		
		\bib{HLM85}{article}{
		author={Hilden, Hugh~M.},
		author={Lozano, Mar{\'i}a~Teresa},
		author={Montesinos, Jos{\'e}~Mar{\'i}a},
		title={On knots that are universal},
		date={1985},
		ISSN={0040-9383},
		journal={Topology},
		volume={24},
		number={4},
		pages={499\ndash 504},
		}
		
		\bib{Hir74}{article}{
		author={Hirsch, Ulrich},
		title={Über offene {{Abbildungen}} auf die 3-{{Sph\"are}}},
		date={1974},
		ISSN={0025-5874},
		journal={Math. Z.},
		volume={140},
		pages={203\ndash 230},
		}
		
		\bib{HW17}{article}{
			author={Huang, Zheng},
			author={Wang, Biao},
			title={Closed minimal surfaces in cusped hyperbolic three-manifolds},
			date={2017},
			ISSN={0046-5755},
			journal={Geom. Dedicata},
			volume={189},
			pages={17\ndash 37},
		}
		
		\bib{JS79}{article}{
			author={Jaco, William~H.},
			author={Shalen, Peter~B.},
			title={Seifert fibered spaces in 3-manifolds},
			date={1979},
			ISSN={0065-9266},
			journal={Mem. Amer. Math. Soc.},
			volume={21},
			number={220},
			pages={viii+192},
		}
		
		\bib{Kap01}{book}{
			author={Kapovich, Michael},
			title={Hyperbolic manifolds and discrete groups},
			series={Progress in {{Mathematics}}},
			publisher={{Birkh\"auser Boston, Inc., Boston, MA}},
			date={2001},
			volume={183},
			ISBN={978-0-8176-3904-4},
		}
		
		\bib{MR03}{book}{
			author={Maclachlan, Colin},
			author={Reid, Alan~W.},
			title={The arithmetic of hyperbolic 3-manifolds},
			series={Graduate Texts in Mathematics},
			publisher={{Springer-Verlag, New York}},
			date={2003},
			volume={219},
			ISBN={0-387-98386-4},
		}
		
		\bib{Mar74}{article}{
			author={Marden, Albert},
			title={The geometry of finitely generated {Kleinian} groups},
			date={1974},
			ISSN={0003-486X},
			journal={Ann. of Math. (2)},
			volume={99},
			pages={383\ndash 462},
		}
		
		\bib{MY82}{article}{
			author={Meeks, William~H., III},
			author={Yau, Shing~Tung},
			title={The classical {{Plateau}} problem and the topology of
					three-dimensional manifolds},
			date={1982},
			ISSN={0040-9383},
			journal={Topology},
			volume={21},
			number={4},
			pages={409\ndash 442},
		}
		
		\bib{Mon74}{article}{
			author={Montesinos, Jos{\'e}~M.},
			title={A representation of closed orientable 3-manifolds as 3-fold
					branched coverings of {$S^3$}},
			date={1974},
			ISSN={0002-9904},
			journal={Bull. Amer. Math. Soc.},
			volume={80},
			pages={845\ndash 846},
		}
		
		\bib{Mye82}{article}{
			author={Myers, Robert},
			title={Simple knots in compact, orientable 3-manifolds},
			date={1982},
			ISSN={0002-9947},
			journal={Trans. Amer. Math. Soc.},
			volume={273},
			number={1},
			pages={75\ndash 91},
		}
		
		\bib{Rol76}{book}{
			author={Rolfsen, Dale},
			title={Knots and links},
			series={Mathematics {{Lecture Series}}, {{No}}. 7},
			publisher={{Publish or Perish, Inc., Berkeley, Calif.}},
			date={1976},
		}
		
		\bib{SU81}{article}{
			author={Sacks, J.},
			author={Uhlenbeck, K.},
			title={The existence of minimal immersions of 2-spheres},
			date={1981},
			ISSN={0003-486X},
			journal={Ann. of Math. (2)},
			volume={113},
			number={1},
			pages={1\ndash 24},
		}
		
		\bib{SU82}{article}{
			author={Sacks, J.},
			author={Uhlenbeck, K.},
			title={Minimal immersions of closed {{Riemann}} surfaces},
			date={1982},
			ISSN={0002-9947},
			journal={Trans. Amer. Math. Soc.},
			volume={271},
			number={2},
			pages={639\ndash 652},
		}
		
		\bib{SY79}{article}{
			author={Schoen, R.},
			author={Yau, Shing~Tung},
			title={Existence of incompressible minimal surfaces and the topology of
					three-dimensional manifolds with nonnegative scalar curvature},
			date={1979},
			ISSN={0003-486X},
			journal={Ann. of Math. (2)},
			volume={110},
			number={1},
			pages={127\ndash 142},
		}
		\bib{Sta65}{article}{
			author={Stallings, John},
			title={Homology and central series of groups},
			journal={J. Algebra},
			volume={2},
			date={1965},
			pages={170\ndash 181},
		}
		
		\bib{Swa71}{article}{
			author={Swan, Richard~G.},
			title={Generators and relations for certain special linear groups},
			date={1971},
			ISSN={0001-8708},
			journal={Advances in Math.},
			volume={6},
			pages={1\ndash 77 (1971)},
		}
		
		\bib{Thu82}{article}{
			author={Thurston, William P.},
			title={Universal links},
			date={1982},
			journal={preprint},
		}
	\end{biblist}
\end{bibdiv}
\end{document}